\documentclass[a4paper]{article}
\usepackage[margin=2.5cm]{geometry} 
\usepackage{tikz}
\usepackage{subcaption} 
\usepackage{pgfplots} 
\usetikzlibrary{shapes,arrows.meta,positioning,calc}
\usepackage{float}
\pgfplotsset{compat=1.17}
\usepackage{graphicx} 
\usepackage[english]{babel}
\usepackage{amsmath}
\usepackage{amsthm}
\usepackage{graphicx}
\usepackage[colorlinks=true, allcolors=blue]{hyperref}
\usepackage{amsfonts}
\usepackage{cases}
\usepackage{bm}
\usepackage[utf8]{inputenc}
\usepackage{tikz}
\usepackage{amssymb}
\newtheorem{lemma}{Lemma}[section]
\newtheorem{proposition}{Proposition}[section]

\newtheorem{theorem}{Theorem}[section]
\newtheorem{remark}{Remark}

\newtheorem{corollary}{Corollary}[section]

\title{The Ricci flow with prescribed  curvature on graphs}
\author{Yong Lin, Shuang Liu}
\date{}

\begin{document}
\maketitle
\begin{abstract}
In this paper,  we consider the Ricci flow with prescribed curvature on the finite graph $G=(V,E)$. For any $e\in E,$
\begin{equation*}\label{flow-equation3}
    \frac{d}{dt}\omega(t,e)=-(\kappa(t,e)-\kappa^*(e))\omega(t,e),~~ t>0,
\end{equation*}
where $\omega$ is the weight function, $\kappa$ is Lin-Lu-Yau Ricci curvature, and $\kappa^*$ is the prescribed curvature.
By imposing invariance of the graph distance with respect to time $t$, the Ricci flow introduced above characterizes the weight evolution governed by the Lin–Lu–Yau curvature. We first establish the existence and uniqueness of the solution to this equation on general graphs. Furthermore, for graphs with girth of at least $6$, we prove that the Ricci flow converges exponentially to weights of $\kappa^*$  if and only if $\kappa^*$ is attainable (namely, there exist weights realizing $\kappa^*$). 
In particular, we prove that the weights for constant curvature exist
if and  only if 
\[\max_{\emptyset\neq\Omega \subsetneqq V}\frac{|E(\Omega)|}{|\Omega|} < \frac{|E|}{|V|},\]
where $E(\Omega)$ denotes the set of edges within the induced subgraph of $\Omega$, and $|A|$ is the cardinality of the set $A$. Viewing edge weights as metrics on surface tilings with girth of at least $6$ or the duals of triangulations with vertex degrees exceeding 5, we demonstrate that the constant Lin-Lu-Yau curvature flow serves as an analog to the 2D combinatorial Ricci flow for piecewise constant curvature metrics, thereby providing an affirmative answer to Question 2 posed by Chow and Luo (J Differ Geom, 63(1) 2002).
\end{abstract}
\noindent \textbf{Keywords:} Ricci flow;  Lin-Lu-Yau curvature; prescribed curvature; surface tiling

\section{Introduction}
In 1982, Hamilton \cite{H82} introduced the Ricci flow, defined by the evolution equation \[\frac{\partial g_{ij}}{\partial t} = -2R_{ij},\]
where \( g_{ij} \) is the metric tensor and $R_{ij}$ is Ricci curvature tensor on  a manifold $(M,g(t))$. He demonstrated that on a 3-dimensional manifold with positive curvature, the flow converges to a spherical metric. Perelman \cite{P02,2003Finite,2003Ricci} introduced the entropy functional and developed a surgery method to address the formation of singularities, leading to definitive proofs of the Poincaré conjecture and the Geometrization conjecture. 
With the advancement of its efficacy in highlighting the structural features of graphs and networks, extending smooth flows to discrete spaces has garnered significant attention and extensive research.
For example, combinatorial Ricci flows based on discrete Gaussian curvature have been developed extensively from both theoretical and practical perspectives \cite{2002Combinatorial,gu2018discrete,luo2004combinatorial}.

In this paper, we investigate the discrete Ricci curvature flow based on Ollivier curvature on graphs. Unlike the triangulations on compact surfaces without boundary, the graph consists solely of a vertex set and an edge set, and does not contain faces. This discrete version of Ricci flow was originally introduced by Ollivier \cite{2007Ricci},
\begin{equation}\label{ollivier0}
    \frac{d}{dt}d(x,y)=-\kappa(x,y)d(x,y),
\end{equation}
which characterizes the evolution of the metric  ($d$) governed by Ollivier Ricci curvature ($\kappa$). 
Just as classical Ricci flow seeks to reveal the underlying topological essence of a manifold by smoothing out its geometric structures, its discrete counterpart holds potential for characterizing the structural properties of graphs. Practical evidence suggests that the time-discretization of discrete Ricci curvature flow \eqref{ollivier0} can be effectively applied to community detection \cite{2025Discrete,2022Normalized,2019Community,2025Community}, network alignment \cite{2018Network}, and deep learning \cite{2024Deep}, etc. These applications of Ricci flow have also been extensively utilized across diverse disciplines, including Biology and Finance (see, for example, \cite{2023Charting,Sh25}).

In recent years, research focus has shifted toward theoretical foundations, leading to extensive investigations into the existence, uniqueness, and convergence of the solution to \eqref{ollivier0}. Specifically,  Bai et al. \cite{2024Ollivier} established the existence and uniqueness of the solution to \eqref{ollivier0}  under the assumption  that the edge $uv$ is the shortest path from $u$ to $v$, while introducing an edge-removal operation. To circumvent the need for such surgery, Ma and Yang \cite{MY24} developed a modified curvature flow. Subsequently, numerous related studies have extended this Ricci flow  to multi-step Ollivier curvature \cite{2024Evolution} and hypergraphs \cite{2025Community}. Regarding the convergence of the Ricci flow, it is further complicated by the lack of an explicit expression for Ollivier curvature and the inherent nonlinearity of the Ricci flow. Bai et al. \cite{BHLL25} investigated the convergence on trees, while Ma and Yang \cite{MY25} simplified the evolution equation to study the convergence of piecewise linear Ricci flows.
Notice that all of these works assume that the distance is a function of the weights. 

 However, on a graph, the distance and the weights can be defined independently. In \cite{BLL25}, Bai et al. considered the evolutionary relationship between distance and curvature while the weights remain time-invariant, and the well-posedness and convergence of the curvature flow were studied. Li and Münch \cite{RF24} investigated the discrete-time Ollivier flow in this setting and established the convergence of the flow under surgery. In contrast, this paper investigates the inverse process: we study the evolutionary dynamics between weights and curvature while keeping the distances time-invariant. Furthermore, given a prescribed curvature $\kappa^*$, we consider the following Ricci flow
\begin{equation}\label{flow-equation}
    \frac{d}{dt}\omega(t,e)=-(\kappa(t,e)-\kappa^*(e))\omega(t,e),~~ t>0,
\end{equation}
where $\omega(t,e)$ and $\kappa(t,e)$ are the weight and Lin-Lu-Yau curvature (an adaptation of  Ollivier curvature) of edge $e$ at time $t$, respectively. This  Ricci  flow describes the evolution of weights on a graph over time, such that the Lin-Lu-Yau curvature becomes uniform or reaches an ideal state. Equations of this form are also known as prescribed curvature flows and can be viewed as a normalized form of the original form with $\kappa^*(e)\equiv0$ (see \eqref{flow-equation0} below). 

In this paper, we first establish the existence and uniqueness of a solution to \eqref{flow-equation} for any prescribed curvature by proving that the curvature is locally Lipschitz, see Theorem \ref{main1}. 
Furthermore,  on graphs with girth of at least $6$, by transforming the Ricci flow into a gradient flow, we  prove the exponential convergence of the solution to \eqref{flow-equation}. We say that the prescribed curvature $\kappa^*$ is attainable if there exists $\bm \omega^*=(\omega_{e_1},\cdots,\omega_{e_m})$ such that $\kappa(\bm \omega^*)=\kappa^*$. And, $\bm \omega^*$ is the weight of $\kappa^*$.
\begin{theorem} [Theorem \ref{main2} and Theorem \ref{main3}]
On a graph with girth of at least 6, 
the Ricci flow \eqref{flow-equation} with a positive initial value converges exponentially to the weight of the prescribed curvature $\kappa^*$ if and  only if $\kappa^*$ is attainable. \label{them:1} 
\end{theorem}

A natural question is whether, for a given curvature, there exists a corresponding weight with this prescribed curvature. 
Given that the total curvature remains constant on graphs with girth of at least $6$ (see Proposition \ref{pro-curvature}),  we provide a necessary and sufficient condition 
\begin{equation}\label{ns}
    \max_{\emptyset\neq\Omega \subsetneqq V}\frac{|E(\Omega)|}{|\Omega|} < \frac{|E|}{|V|}
\end{equation}
for the existence of a constant curvature weight, see Theorem \ref{condition_th}. In these graphs, $\kappa^*$ in \eqref{flow-equation} is naturally chosen to be the average curvature 
\[\bar{\kappa}:=\frac{\sum_{e\in E}\kappa_e}{|E|}=2\left(\frac{|V|}{|E|}-1\right)\]
to get
\begin{equation}\label{flow-equation1}
    \frac{d}{dt}\omega(t,e)=-(\kappa(t,e)-\bar{\kappa})\omega(t,e),~~ t>0,
\end{equation}
which can be regarded as a normalized version of the original form 
\begin{equation}\label{flow-equation0}
    \frac{d}{dt}\omega(t,e)=-\kappa(t,e)\omega(t,e),~~ t>0.
\end{equation}
As a consequence, we demonstrate the exponential convergence of the solutions to \eqref{flow-equation1} and then \eqref{flow-equation0}.

\begin{theorem} [Corollary \ref{main_coro}]
      On a graph with girth of at least 6, 
the Ricci flow \eqref{flow-equation1} with a positive initial value converges exponentially to the weight $\bm \omega$ of constant curvature $\bar{\kappa}$ if and  only if the condition \eqref{ns} holds. Additionally, the un-normalized Ricci flow \eqref{flow-equation0} drives the curvature to approach the average curvature $\bar{\kappa}$ exponentially. \label{them:1.1}
\end{theorem}

Driven by constant curvature, the Ricci flow \eqref{flow-equation1} evolves edge weights so that bridges acting as bottlenecks are assigned significantly higher weights (see the dumbbell graph $D_{6,6}$ in  Section \ref{simulations}). This evolution effectively identifies bottlenecks and assigns optimal weights in networks with girth of at least 6. 
While the evolution of the previous Ollivier Ricci flow (see e.g.  \cite{MY24} and \cite{2019Community}) can also identify network structures, our approach primarily focuses on constant curvature-driven weights.

From the perspective of polygonal partitions of surfaces, Theorem \ref{them:1.1} can generally be referred to the Discrete Uniformization Theorem, serving as the discrete analogue to the classical Riemann Uniformization Theorem for manifolds.
The study of this problem   originates from Thurston’s work on circle packing. Thurston \cite{thurston1978geometry} introduced triangulated surfaces endowed with circle packing (CP) metrics and proved the existence conditions for constant discrete Gaussian curvature (defined by angle deficit) metrics; subsequently, the combinatorial Ricci flow developed by Chow and Luo \cite{2002Combinatorial} provides an algorithm that allows any initial metric to converge through the Ricci flow towards the constant curvature metric. In addition to CP metrics, Luo \cite{luo2004combinatorial}  also adopted piecewise linear (PL) metrics on surface triangulations, establishing the necessary and sufficient conditions for the existence of PL metrics with constant curvature and investigating the convergence of solutions for the combinatorial Yamabe flow with surgeries. Discrete Calabi flow \cite{ge2018combinatorial} has also been introduced to address a broader range of geometric problems.  Combinatorial Ricci flows have received considerable attention in the literature. For further details, the reader is referred to \cite{ge2021circle,ge2021combinatorial,ge2024character,gu2018discrete}.

In this paper, we employ the Lin-Lu-Yau curvature flow, which is fundamentally distinct from the  discrete Gaussian curvature. Lin-Lu-Yau curvature offers a simpler alternative to discrete Gaussian curvature as it operates exclusively on vertices and edges, bypassing the need for higher-dimensional face information. Specifically, in contrast to Luo’s condition in \cite{luo2004combinatorial} for the existence of a PL metric with constant discrete Gaussian curvature, our condition \eqref{ns} for the existence of weights yielding constant Lin-Lu-Yau curvature is less restrictive when viewed from the perspective of the dual tessellation, see Corollary \ref{coro_sur}.
However, the Lin-Lu-Yau  curvature flow still plays a role in polygonal partitions of surfaces. Perform a polygonal tiling of the surface (excluding triangles, quadrilaterals, and pentagons) and extract the 1-skeleton of the tiling. The initial edge weights are set as the geodesic distance between the endpoints of the edges. The Ricci flow \eqref{flow-equation1} eventually evolves these edge weights into those corresponding to constant curvature. Notice that for a regular graph or a semi-regular bipartite graph with constant curvature, the weights are also constant, see Proposition \ref{regular}.  The constant weights can still be treated as edge metrics. 

\begin{corollary} [Proposition \ref{regular} and Corollary \ref{main_coro}]
Let $(T, \bm{\omega}_0)$ be a polygonal tiling of a closed connected surface $X$, where $\bm{\omega}_0$ denotes the geodesic distance between the two endpoints of the set of edges. Assume that the 1-skeleton $G$ of $T$ is a regular or semi-regular bipartite graph with girth greater than 5. Then the solution to  \eqref{flow-equation1} on $G$ converges to a constant value, which corresponds to a polygonal surface partition with constant edge lengths. 
\end{corollary}
For instance, on a surface with a hexagonal circle packing, the Ricci flow \eqref{flow-equation1} transforms the  hexagonal grid with distorted edge lengths into a regular hexagonal grid where all edge lengths are uniform.
This result serves as an analogue for hexagonal tilings on the torus of the conclusion by Rodin and Sullivan \cite{rodin1987convergence} , which states that a circle packing in the plane with the hexagonal pattern is the regular hexagonal packing.
The simulation procedure for the Ricci flow \eqref{flow-equation1}  is provided in Section \ref{simulations}.
In fact, this is not limited to regular and semi-regular bipartite graphs; as long as the converged result satisfies the polygon closure condition, i.e., the length of the longest edge must be less than the sum of the remaining edges in any cycle, it can be treated as an edge metric and applied to surface polygonal tessellation. This provides a positive answer to Question 2 (concerning the 2-dimensional combinatorial Ricci flow for piecewise constant curvature metrics) proposed by Chow and Luo \cite{2002Combinatorial}.  When combinatorial curvature flows  investigate the evolutionary relationship between PL metrics and combinatorial curvature \cite{zglg12}, their evolution often leads to the failure of the triangle inequality for edge metrics.  A distinct advantage of \eqref{flow-equation1} lies in investigating the evolutionary relationship between edge weights and edge curvature; it does not require the polygon closure condition to be maintained throughout the process, provided that the initial and converged values satisfy this condition.

The structure of the paper is as follows: In Section \ref{section2},  the definitions of Lin-Lu-Yau curvature (flow) are introduced, and the existence and uniqueness of the solution to the curvature flow  with prescribed curvature on general graphs are established. In Section \ref{section3}, we only consider graphs with girth at least $6$, provide the necessary and sufficient conditions for the existence of constant curvature and prove the convergence of the Ricci flow with the prescribed curvature. In Section \ref{simulations}, we present the applications and numerical simulations of the Ricci flow with the prescribed constant curvature.

\section{Lin-Lu-Yau curvature (flow)  on graphs}\label{section2}
Let $G=(V,E,\omega)$ be a weighted graph, where $V$ is the set of vertices, $E\subset V\times V$ is the set of edges, and $\omega: E\rightarrow\mathbb{R^+}$ is the weight function. In this paper, we only consider finite and connected graphs. 
Let $\Omega \subseteq V$ be a subset of vertices. Let $E(\Omega)$ denote the set of edges in the subgraph induced by $\Omega$. We denote the cardinalities of the vertex set and the edge set of this induced subgraph by $|\Omega|$ and $|E(\Omega)|$, respectively. The set of edges connecting $\Omega$ and its complement $\Omega^c$ is denoted by $E(\Omega, \Omega^c)$. Furthermore, the distance $d(u, v)$ between two vertices $u, v \in V$ is defined as the length of the shortest path connecting them. Denote 
\[d(x):=\#\{y\in V: y\sim x\}\]
by  the degree of $x\in V$, 
where $y\sim x$ means that $(x,y)\in E.$ Let $\mathbb{R}^V$ be the set of real function on $V$.

\subsection{Lin-Lu-Yau curvature}
Lin-Lu-Yau  curvature is a modified version of the Ollivier Ricci curvature \cite{2007Ricci}, designed to measure the discrepancy between the wasserstein distance and the graph distance. Let \( x\neq y\in V \), \( \mu_x\) and \( \mu_y \) are two probability distributions defined on $V$. The wasserstein distance $W(\mu_x,\mu_y)$ is defined by 
\[W(\mu_x,\mu_y)=\inf_{A \in \Pi(\mu_x, \mu_y)} \sum_{u,v \in V} d(u,v) A(u,v),\]
where \( A : V \times V \to [0, 1] \) is the transport plan from \( \mu_x \) to \( \mu_y \), satisfying
\[
\begin{cases}
\sum_{v \in V} A(u,v) =  \mu_x (u), & u \in V, \\
\sum_{u \in V} A(u,v) =  \mu_y (v), & v \in V,
\end{cases}
\]
and the minimum is taken over all transport plans from \(  \mu_x \) to \(  \mu_x \).

For $\alpha > 0$,  a finitely supported probability measure is denoted by
\begin{equation*}
    m_x^\alpha(y):=\left\{\begin{aligned}
    &\alpha,&~~ y=x,\\
    &(1-\alpha)\frac{\omega_{xy}}{m(x)},&~~y\sim x,\\
    &0,&~~\mbox{otherwise,}
\end{aligned}
  \right.
\end{equation*}
where $m(x):=\sum_{y\sim x}\omega_{xy}$.
 Lin-Lu-Yau curvature is defined by,   for $x\neq y$ 
\[\kappa(x,y):=\lim_{\alpha\rightarrow 1^-}\frac{1}{1-\alpha}\left(1-\frac{W(m_x^\alpha,m_y^\alpha)}{d(x,y)}\right),\]
where $W(m_x^\alpha,m_y^\alpha)$ denotes the wasserstein distance between $m_x^\alpha$ and $m_y^\alpha$.
One limit-free form of Ollivier-Ricci curvature was proposed by Münch and Wojciechowski \cite{Florentin2017Ollivier}. Denote
      \[\nabla_{xy}f := \frac{f(x) - f(y)}{d(x, y)}\]
      by the gradient of $f$ with respect to $x$ and $y.$
Let $\|\nabla f\|_\infty=\sup_{x,y\in V}|\nabla_{xy}f|$. For $K \geq 0$,
\[
\operatorname{Lip}(K) := \{ f \in C(V) : \|\nabla f\|_\infty \leq K \}
\]
is the set of all $K$-Lipschitz functions on $V$ with respect to the
graph metric $d$. Then, Lin-Lu-Yau curvature can be re-expressed as
\begin{equation}\label{ollivier}
\kappa(x,y)=\inf_{f\in \mathcal{F} }\nabla_{xy}\Delta f,
\end{equation}
where $\mathcal{F}:=\{f\in \operatorname{Lip}(1),\nabla_{yx}f=1\}$, and the Laplacian is defined by 
\[\Delta f(x)=\frac{1}{m(x)}\sum_{z\sim x}\omega_{xz}(f(z)-f(x)),~~\forall x\in V.\]

Notice that when $f\in \operatorname{Lip}(1)$, we have for any $x\in V,$
$$|\Delta f(x)|\leq\frac{1}{m(x)}\sum_{z\sim x}\omega_{xz}|f(z)-f(x)|\leq \frac{1}{m(x)}\sum_{z\sim x}\omega_{xz}d(x,z)=1$$
From it, we have the following bounds of  $\kappa$.
\begin{lemma} 
The Ricci curvature $\kappa(x,y)$ for any $x,y\in V$ is bounded. Specifically,
\[-\frac{2}{d(x,y)}\leq\kappa(x,y)\leq \frac{2}{d(x,y)}.\]
In particular, $\kappa_e\in[-2,2]$ for any $e\in E$.
\end{lemma}\label{bounds}
\begin{proof} For any $f\in \operatorname{Lip}(1)$,
\[|\nabla_{xy}\Delta f|= \left|\frac{\Delta f(x) - \Delta f(y)}{d(x, y)}\right|\leq \frac{|\Delta f(x)| + |\Delta f(y)|}{d(x, y)}\leq \frac{2}{d(x,y)},\] 
it follows that the bounds of $\kappa(x,y)$.
\end{proof}

\subsection{The Ricci flow with prescribed curvature}
In this subsection, we consider the following Lin-Lu-Yau Ricci flow  with prescribed curvature:
for any $e\in E$,
\begin{equation*}\label{flow-equation3}
    \frac{d}{dt}\omega(t,e)=-(\kappa(t,e)-\kappa^*(e))\omega(t,e),~~ t>0,
\end{equation*}
with a positive initial value $\omega(t,e)=\omega_0(e),$ where $\kappa^*(e)$ is a prescribed curvature. 
Let $E=\{e_1,\cdots,e_n\}$, and $\bm \omega =(\omega_1,\omega_2,\cdots,\omega_n)$ with $\omega_i=\omega(e_i)$. Then the Ricci flow with the prescribed curvature  can be rewritten as 
\begin{equation}\label{case_a} 
    \begin{cases}
\frac{d}{dt}\omega_i=-(\kappa_i(\bm \omega)-\kappa_i^*)\omega_i  \\
\bm \omega(0)=\bm \omega_0
\end{cases}
\end{equation}
with a positive initial value $\bm \omega_0=(\omega_{1,0},\omega_{2,0},\cdots,\omega_{n,0})$,
where  $\bm \kappa^*=(\kappa_1^*,\kappa_2^*,\cdots,\kappa_n^*)$ is the prescribed curvature vector.
The following results were established for directed graphs in our earlier work \cite{2025Ricci}. For the sake of completeness, we provide the proof for the case of undirected graphs below.
\begin{lemma}\label{local_lip}
The Ricci curvature $\kappa_i$ for any $i$ is locally Lipschitz continuous  on \[A:= \{ \boldsymbol{\omega} \in \mathbb{R}^{n} \mid \omega_i > 0,  \forall i\}.\]
\end{lemma}
\begin{proof}
Since the curvature is scale-invariant with respect to $\bm \omega$, that is, the curvature remains unchanged when the weight vector is multiplied by a positive constant.
Let $\bm \omega\neq c\bm \omega'$ with $c>0$, and there exist $\delta>0$ such that 
\[\delta^{-1}\leq\omega_i\leq \delta, \quad \forall i.\]
Let $f$ and $f'$ be the optimal functions that attain the infimum in \eqref{ollivier} for $\bm \omega$ and $\bm \omega'$, respectively. Let $\Delta'$ be the operator associated with $\bm \omega'$. Thus, for  $e_i=(x,y)$, 
  \begin{align*}
      |\kappa_i(\bm \omega)-\kappa_i(\bm \omega')|
      &=|\nabla_{xy}\Delta f-\nabla_{xy}\Delta' f'|\\
      &\leq \max\{|\nabla_{xy}\Delta f-\nabla_{xy}\Delta' f|,|\nabla_{xy}\Delta f'-\nabla_{xy}\Delta'f'|\}.
  \end{align*}
Due to
\begin{align*}|\nabla_{xy}\Delta f-\nabla_{xy}\Delta' f|
&=|\Delta f(x)-\Delta f(y)-\Delta'f(x)+\Delta' f(y)|\\
&\leq\sum_{u\sim x}\left|\frac{\omega_{xu}}{m(x)}-\frac{\omega'_{xu}}{m'(x)}\right||f(u)-f(x)|+\sum_{v\sim y}\left|\frac{\omega_{yv}}{m(y)}-\frac{\omega'_{yv}}{m'(y)}\right||f(v)-f(y)|\\
&\leq\sum_{u\sim x}\frac{|m'(x)\omega_{xu}-m(x)\omega'_{xu}|}{m(x)m'(x)}+\sum_{v\sim y}\frac{|m'(y)\omega_{yv}-m(y)\omega'_{yv}|}{m(y)m'(y)},
\end{align*}
and 
\[
    |m'(x)\omega_{xu}-m(x)\omega'_{xu}|\leq m'(x)|\omega_{xu}-\omega'_{xu}|+\omega'_{xu}|m'(x)-m(x)|\leq 2\delta |E|\|\bm \omega-\bm\omega'\|_\infty,
\]
where $\|\bm \omega-\bm\omega'\|_\infty=\sup_{i}|\omega_i-\omega_i'|$. 
We have
\[|\nabla_{xy}\Delta f-\nabla_{xy}\Delta' f|\leq 4\delta^3 |E|^2\|\bm \omega-\bm\omega'\|_\infty.\]
Similarly, $|\nabla_{xy}\Delta f'-\nabla_{xy}\Delta'f'|$ admits the same bound.
Therefore,
\[|\kappa_i(\bm \omega)-\kappa_i(\bm \omega')|\leq 4\delta^3 |E|^2\|\bm \omega-\bm\omega'\|_\infty,\]
which completes the proof.
\end{proof}

\begin{theorem}  \label{main1}
Equations \eqref{case_a} admit a unique positive solution on $[0,\infty)$ with the positive initial conditions  \(\boldsymbol{\omega}_0\).
\end{theorem}

\begin{proof}
Let $\boldsymbol{\omega}\neq c \boldsymbol{\omega}'$, $c\in \mathbb{R}$. Due to
\begin{equation*}
\begin{aligned}
|(\kappa_i(\boldsymbol{\omega}) -\kappa_i^*)\omega_i - (\kappa_i(\boldsymbol{\omega'}) -\kappa_i^*)\omega_i'| 
&= |\kappa_i(\boldsymbol{\omega})\omega_i - \kappa_i(\boldsymbol{\omega}')\omega_i'-\kappa_i^*(\omega_i-\omega_i')| \\
&\le \left|\kappa_i(\boldsymbol{\omega}) - \kappa_i(\boldsymbol{\omega}')\right|\cdot \omega_i + \left(|\kappa_i(\boldsymbol{\omega}')|+\max_i|\kappa^*_i|\right) \, \|\boldsymbol{\omega} - \boldsymbol{\omega}'\|_\infty\\
\end{aligned}
\end{equation*}
Due to Lemma \ref{bounds} and Lemma \ref{local_lip}, the Picard-Lindelöf theorem guaranties a time \(T > 0\) such that a unique solution to \eqref{case_a} exists on \([0, T]\).

Next, we demonstrate the long time existence. Let 
$$T_{\max}:=\max\{T:\text{Equations \eqref{case_a}   has a solution on~} [0,T]\}.$$
Assume that \( T_{\max} < \infty \).  By the extension theorem for ODEs, as \( t \rightarrow T_{\max}^{-} \), \(\boldsymbol{\omega}(t)\) must either 
  \begin{enumerate}
    \item escape to infinity, i.e., \(\|\boldsymbol{\omega}(t)\| \rightarrow \infty\), or
    \item approach the boundary \(\partial A = \{ \boldsymbol{\omega} \mid \min_i \omega_i = 0 \}\).
  \end{enumerate}
By Lemma \ref{bounds}, \(|\kappa_i(\boldsymbol{\omega}) - \kappa_i^*| \leq M\) where \(M = 2+\max_i|\kappa_i^*|\).  
It follows that for any $t\in[0,T]$

  \[
    \omega_i(t) \geq \omega_i(0) e^{-M T_{\max}} > 0,  \]
    and
    \[
    \omega_i(t)  \leq \omega_i(0) e^{M T_{\max}} <\infty .
  \]

\item Thus, \(\boldsymbol{\omega}(t)\) neither escapes to infinity nor reaches \(\partial A\), contradicting \( T_{\max} < \infty \).
This completes the proof.
\end{proof}
\section{Graphs with girth at least $6$}\label{section3}
In this section, we only consider the graph not contained in any $3-,4-$ or $5-$cycles. For \(e=(x, y)\), Lin-Lu-Yau curvature is
\begin{equation}\label{cur_tree}
    \kappa_e=2\omega_e\left(\frac{1}{m(x)}+\frac{1}{m(y)}\right)-2
\end{equation}
by choosing $f$ in \eqref{ollivier} as 
\[f(z) =
\begin{cases}
0 & : z \sim x \text{ and } z \neq y ,\\
1 & : z = x ,\\
2 & : z = y, \\
3 & : z \sim y \text{ and } z \neq x,
\end{cases}\]
see \cite{Florentin2017Ollivier}  for details.

\subsection{The existence of weights for constant curvature}

\begin{proposition}\label{pro-curvature}
Assume that $\omega_e>0$ for any $e\in E$. Then
\begin{itemize}
    \item $\kappa_e\in (-2,2]$, and $\kappa_e=2$ if and only if the graph is $K_2$, that is, a graph with two vertices and one edge.
    \item $\sum_{e\in E}\kappa_e=2(|V|-|E|)$.
\end{itemize}
\end{proposition}
\begin{proof} 
By \eqref{cur_tree}, the first property is obvious. As for the second one,
\[
\sum_{e\in E}\kappa_e= 2 \sum_{e\in E} \omega_e\left(\frac{1}{m(x)}+\frac{1}{m(y)}\right)- 2 |E|.
\]
The summation on the right hand side can be regrouped as
\[
\sum_{e\in E} \omega_e\left(\frac{1}{m(x)}+\frac{1}{m(y)}\right) = \sum_{z \in V} \sum_{e_z\in N(z)} \omega_{e_z}\frac{1}{m(z)},
\]
where \( N(z)\) is the set of edges incident to node \(z\). For each node \(z\), we have
\[
\sum_{e_z\in N(z)} \omega_{e_z}\frac{1}{m(z)} = 1.
\]
Therefore,
\[
\sum_{e\in E} \omega_e\left(\frac{1}{m(x)}+\frac{1}{m(y)}\right) = |V|.
\]
That completes the proof.
\end{proof}

Based on the second property, we can define the average curvature on graphs with girth of at least 6 as
\[\bar{\kappa}:=\frac{\sum_{e\in E}\kappa_e}{|E|}=2\left(\frac{|V|}{|E|}-1\right).\]
\begin{remark}
Based on the expression for average curvature, we can classify graphs with girth at least 6 by the sign of the average curvature, which is determined by comparing the number of vertices and edges.
\begin{itemize}
    \item For any tree, due to $|V|=|E|+1$, we have 
    \(\bar{\kappa}=\frac{2}{|V|-1}(>0).\)
    \item For any unicyclic graph containing exactly one cycle of length at least $6$, i.e. $|V|=|E|$, we have 
    \(\bar{\kappa}=0\).
    \item For any graph containing more than one cycle of length at least $6$, i.e. $|V|<|E|$, thus 
   $ \bar{\kappa}< 0.$ 
\end{itemize}
\end{remark}

\begin{proposition} \label{injective}
Let $G=(V,\omega)$ be a graph with girth at least $6$. The map $\Pi:\mathbb{R}^n_{>0}\rightarrow \mathbb{R}^n$, which sends a weight $\bm \omega$ to its Ricci curvature $\bm \kappa$, when restricted to the hypersurface 
\[U:=\{(\omega_1,\cdots, \omega_n)\in R^n_{>0} |\sum_{i=1}^n\omega_i = 1\}\] 
is injective. That is, the weight $\bm \omega$  is determined, up to a scalar multiplication, by the curvature $\bm \kappa$.
\end{proposition}
\begin{proof}
Assume that $\kappa_e(\bm\omega)=\kappa_e(\bm \omega')$, which implies
\[\omega_e\left(\frac{1}{\sum_{u\sim x} \omega_{xu}}+\frac{1}{\sum_{v\sim y} \omega_{yv}}\right)=\omega_e'\left(\frac{1}{\sum_{u\sim x} \omega_{xu}'}+\frac{1}{\sum_{v\sim y} \omega_{yv}'}\right),\]
where $\sum_{e\in E}\omega_e=\sum_{e\in E}\omega_e'=1$. Next, we have to prove that $\omega_e=\omega'_e$ for any $e\in E$. Set 
\[\rho_e=\frac{\omega_e}{\omega_e'}.\]
If there exists $h\in E$ such that $\omega_h\neq\omega_h'$. without loss of generality, we assume that $\rho_{\min}:=\min_{i}\rho_i<1$, and $h=(x_h,y_h)$ is the minimized edge. Thus
\begin{equation}\label{eq1}
\rho_{\min}\left(\frac{1}{\sum_{u\sim x_h} \omega_{x_hu}}+\frac{1}{\sum_{v\sim y_h} \omega_{y_hv}}\right)=\frac{1}{\sum_{u\sim x_h} \omega_{x_hu}'}+\frac{1}{\sum_{v\sim y_h} \omega_{y_hv}'}.
\end{equation}
On the other hand, for any $u\in V$, we have
\[\sum_{v\sim u} \omega_{uv}=\sum_{v\sim u}  \rho_{uv} \omega_{uv}'\geq \rho_{\min}\sum_{v\sim u} \omega'_{uv}.\]
And if there exists $v\in N(u)$ such that $\rho_{uv}>\rho_{\min}$, then
\[\sum_{v\sim u} \omega_{uv}>\rho_{\min}\sum_{v\sim u} \omega'_{uv}.\]
Thus, if there exists $v\in N(x_h)\cup N(y_h)$ such that $\rho_{uv}>\rho_{\min}$, substituting the above inequalities into \eqref{eq1}, we get a contradiction. Based on the connectivity of the graph, we can conclude that $\rho_e\equiv \rho$ with $\rho>0$, which implies that 
\[\omega_e=\rho\omega'_e\]
for any $e\in E.$ By
\[\sum_{e\in E}\omega_e = \sum_{e\in E}\omega'_e=1,\]
we have $\rho=1$. This completes the proof.
\end{proof}
Next, we explore the sufficient and necessary condition for the existence of positive weights for constant curvature, namely, to find the condition that assures there is a set of weights such that
\begin{equation}\label{constant}
    \omega_e\left(\frac{1}{\sum_{u\sim x} \omega_{xu}}+\frac{1}{\sum_{v\sim y} \omega_{yv}}\right)=\frac{|V|}{|E|}
\end{equation}
hold for some \(\omega_e > 0\) for any $e\in E$.
 
\begin{proposition}\label{regular}
There exist constant weights on $E$ satisfying $\kappa_e=\bar{\kappa}$ for any $e\in E$  if and only if the graph is a regular graph or a semi-regular bipartite graph.    
\end{proposition}
\begin{proof}
\textbf{Necessity:} From \eqref{constant}, for any $v\sim u$, $$\frac{1}{d(v)} = \frac{|V|}{|E|}- \frac{1}{d(u)}.$$
For any $w\sim v, w\neq v$, we can conclude that $d(w)=d(u)$. Thus, along any path in the graph, the degrees of the nodes can only alternate between two values. Set $a$ and $b$, respectively.

If $a = b$, this is a regular graph. If $a \neq b$, the nodes are partitioned into two sets. Set $V_a$ with degree $a$ and $V_b$ with degree $b$, respectively. Nodes in $V_a$ can only connect to nodes in $V_b$  and nodes in $V_b$ can only connect to nodes in $V_a$. This is a semi-regular bipartite graph.

\textbf{Sufficiency:} For a $k$-regular graph, the relationship $2|E| = k|V|$ implies 
$$\frac{|V|}{|E|} = \frac{2}{k}.$$
For a $(a, b)$-semi-regular bipartite graph, using  $|E| = |V_a|a = |V_b|b$, we have 
$$\frac{|V|}{|E|} = \frac{|V_a| + |V_b|}{|E|}=\frac{\frac{|E|}{a} + \frac{|E|}{b}}{|E|} = \frac{1}{a} + \frac{1}{b}.$$
In these two cases, Equations \eqref{constant} hold for constant weights. Combining with Proposition \ref{injective}, we can conclude that constant weights are the unique solutions to \eqref{constant} up to a scalar multiplication.
\end{proof}

\begin{theorem} \label{condition_th} 
There exist positive weights on $E$ satisfying $\kappa_e=\bar{\kappa}$ for any $e\in E$ if and only if
\begin{equation}\label{ns-condtion}
\max_{\emptyset\neq\Omega \subsetneqq V}\frac{|E(\Omega)|}{|\Omega|} < \frac{|E|}{|V|}.
\end{equation}
\end{theorem}

\begin{proof}
\textbf{Necessity:} For any $\Omega \subsetneqq V$, 
\begin{align*}
    |\Omega|=\sum_{u \in \Omega} 1 &= \sum_{u \in \Omega} \left( \sum_{v \sim u} \frac{\omega_{uv}}{m(u)} \right)\\
    &=\sum_{u \in \Omega} \sum_{v \in \Omega, v \sim u} \frac{\omega_{uv}}{m(u)} + \sum_{u \in \Omega} \sum_{v \notin \Omega, v \sim u} \frac{\omega_{uv}}{m(u)}\\
    &=\sum_{(u,v)\in E(\Omega)} \left(\frac{\omega_{uv}}{m(u)} +\frac{\omega_{vu}}{m(v)}\right)+\sum_{u \in \Omega} \sum_{v \notin \Omega, v \sim u} \frac{\omega_{uv}}{m(u)}
\end{align*}
where the summation $\sum_{(x,y)\in E(\Omega)}$ is taken over each edge in $E(\Omega)$ exactly once. 
Due to \eqref{constant}, we have
\[\sum_{(u,v)\in E(\Omega)} \left(\frac{\omega_{uv}}{m(u)} +\frac{\omega_{vu}}{m(v)}\right)=\frac{|V|}{|E|}|E(\Omega)|.\]
Moreover,
\[\sum_{u \in \Omega} \sum_{v \notin \Omega, v \sim u} \frac{\omega_{uv}}{m(u)}>0\]
by the connectivity of $G$.
Thus, substituting these two results yields \eqref{ns-condtion}.

\textbf{Sufficiency:}
Firstly, we transform \eqref{constant} to the equivalent system: there exists \(m(x) > 0\) such that
    \begin{equation}\label{equiv}
        \sum_{y \sim x} \frac{m(y)}{m(x) + m(y)} = \frac{|E|}{|V|}, \quad \forall x \in V
    \end{equation}
related to the set of vertices by constructing weights
\[
\omega(x,y) = \frac{|V|}{|E|} \cdot \frac{m(x)m(y)}{m(x) + m(y)}.
\]
Set $g(x)=\ln m(x)$ for any $x\in V.$ Then the equivalent system \eqref{equiv} can be regrouped as 
\begin{equation}\label{new}
    \sum_{y \sim x} \frac{1}{1 + e^{g(x) - g(y)}} =  \frac{|E|}{|V|},\quad \forall x \in V.
\end{equation}
Define the function \(H: \mathbb{R}^{V} \to \mathbb{R}\)
$$H(\bm{g}) = \frac{|E|}{|V|} \sum_{z \in V} g(z) + \frac{1}{2} \sum_{u \in V} \sum_{v \sim u} \psi(g(v) - g(u))- \frac{1}{2} \sum_{z \in V} d(z) g(z), \quad \psi(t) = \ln(1 + e^t).$$
Computing the gradient of \( H(\bm{g}) \) yields 
\begin{align}\label{gradient}
    \frac{1}{2} \frac{\partial}{\partial g(x)} \left( \sum_{u \in V} \sum_{v \sim u} \ln(1 + e^{g(v) - g(u)}) \right) &=\frac{1}{2} \left[ \sum_{y \sim x} \frac{e^{g(x)-g(y)}}{1 + e^{g(x)-g(y)}} - \sum_{y \sim x} \frac{e^{g(y)-g(x)}}{1 + e^{g(y)-g(x)}}  \right]\notag\\
    &= \frac{d(x)}{2}-\sum_{y \sim x} \frac{1}{1 + e^{g(x) - g(y)}}
\end{align}
by 
    \[
    \frac{1}{1 + e^{t}} + \frac{1}{1 + e^{-t}} = 1
    \]
for any $t\in \mathbb{R}$.
Therefore, \eqref{new} are the Euler-Lagrange equations of the functional \( H(\bm{g}) \). 

Notice  that if \( \bm{g} \) is a solution to the system \eqref{new}, so \( \bm{g} + c\bm{1} \) is also a solution for any \( c \in \mathbb{R} \). To eliminate this redundancy, we restrict solutions to the subspace
\[\mathcal{S}:=\{\bm{g}\in \mathbb{R}^V : \sum_{x\in V} g(x) = 0\}.\]
Actually, if $\bm g,\bm g+c\bm 1\in \mathcal{S}$, we can find $c=0.$ 
When minimizing $H(\boldsymbol{g})$ subject to $\sum_{x \in V} g(x) = 0$,
 the Lagrangian is
    \[
    \mathcal{L} = H(\boldsymbol{g}) + \lambda \sum_{x \in V} g(x).
    \]
then 
    \[
    \frac{\partial \mathcal{L}}{\partial g(x)} = \frac{\partial H}{\partial g(x)} + \lambda = 0 \quad \forall x \in V, \qquad \sum_{x \in V} g(x) = 0.
    \]
Moreover,  
\[
  \sum_{x \in V} \frac{\partial H}{\partial g(x)} = |V| \cdot \frac{|E|}{|V|} - \sum_{x \in V} \sum_{y \sim x} \frac{1}{1 + e^{g(x) - g(y)}}= 0
  \]
by
  \[\sum_{x \in V} \sum_{y \sim x} \frac{1}{1 + e^{g(x) - g(y)}}=\sum_{(x,y)\in E} \left(\frac{1}{1 + e^{g(x) - g(y)}}+\frac{1}{1 + e^{g(y) - g(x)}}\right)=|E|.\]
Then, 
summing all components of  the Lagrangian gives
    \[
    0= \sum_{x \in V} \left( \frac{\partial H}{\partial g(x)} + \lambda \right) = 0 + |V| \lambda
    \]
which implies $ \lambda = 0.$ Thus, the conditions simplify to $\nabla H(\boldsymbol{g}) = \bm{0}$ and $\sum_{x \in V} g(x) = 0$.  
    This means that solutions to \eqref{new} correspond exactly to the critical points of $H$ on $\mathcal{S}$. In the following, we shall demonstrate that the restriction of $H$ to $\mathcal{S}$ attains an extreme value. The proof is divided into two steps.

\noindent \textbf{Step 1. $H(\bm{g})$ is strictly convex on $\mathcal{S}$.}\\ From \eqref{gradient}, for any $x\in V$,
$$\frac{\partial^2 H}{\partial g(x)^2} = \sum_{y \sim x} \frac{e^{g(x)-g(y)}}{(1+e^{g(x)-g(y)})^2},$$
 for any $y\sim x$,
$$\frac{\partial^2 H}{\partial g(x) \partial g(y)} = -\frac{e^{g(x)-g(y)}}{(1+e^{g(x)-g(y)})^2}=\frac{\partial^2 H}{\partial g(y) \partial g(x)} $$
and it is $0$ otherwise. Thus,
$$\bm{v}^T (\nabla^2 H) \bm{v} = \sum_{(x,y)\in E} \frac{e^{g(x)-g(y)}}{(1+e^{g(x)-g(y)})^2} (v_x - v_y)^2>0$$
for any $\bm{v}\in \mathcal{S}$ due to $\mathcal{S}\perp\bm 1$.

\noindent\textbf{Step 2. Assume that the condition \eqref{ns-condtion} holds for all \(\Omega \subsetneqq V\). Then, \(H\) is coercive on \(\mathcal{S}\), i.e.,
\(H\rightarrow +\infty ~\mbox{when}~ \|\bm g\|_2\rightarrow +\infty ~\mbox{on}~ \mathcal{S}. \)}\\
 First, we need a classical theorem from Convex Analysis: for a continuous convex function, coercivity of $H$ on $\mathcal{S}$ is equivalent to its asymptotic slope 
 $$\rho(\bm{v}) :=  \lim_{t \to +\infty} \frac{H(t\bm{v}) - H(\bm 0)}{t}$$ being strictly positive for all non-zero directions $\bm{v} \in \mathcal{S} \setminus \{\bm 0\}$ (see, for example, Proposition 3.2.4 and Definition 3.2.5 in \cite{hiriart2004fundamentals}). For any $z\in \mathbb{R}$, using the fact that 
$$\lim_{t \to +\infty} \frac{\ln(1+e^{tz})}{t} = \max(0, z),$$
the asymptotic slope is transformed to 
$$\rho(\bm{v}) =\frac{1}{2} \sum_{x \in V} \sum_{y \sim x}  \max(0, v(y) - v(x)) + \sum_{x \in V} \alpha(x) v(x) = \frac{1}{2} \sum_{(x,y) \in E} |v(x) - v(y)| + \sum_{x \in V} \alpha(x) v(x),$$
where $\alpha(x) = \frac{|E|}{|V|}- \frac{1}{2}d(x)$. 
Let $\bm1_\Omega$ be the indicator function of $\Omega$.
To prove that $\rho(\bm{v}) > 0$ holds for all $\bm{v} \in \mathcal{S}$, we {\it claim} that it suffices to verify that the slope $\rho(\bm1_\Omega)$ is positive for all non-empty proper subsets $\emptyset\neq\Omega \subsetneqq V$, namely,
\begin{equation}\label{key}
    \rho(\bm1_\Omega) = \frac{1}{2}|E(\Omega, \Omega^c)| + \sum_{x \in \Omega} \alpha(x) > 0.
\end{equation}
Next, we prove \eqref{key}. We utilize the fundamental graph identity
\begin{equation}\label{deg}
    \sum_{x \in \Omega} d(x) = 2|E(\Omega)| + |E(\Omega, \Omega^c)|.
\end{equation}
Substituting this yields
$$\rho(\bm1_\Omega)  =  \frac{1}{2}|E(\Omega, \Omega^c)| + \sum_{x \in \Omega} \left(  \frac{|E|}{|V|}- \frac{1}{2}d(x) \right) = |\Omega|\frac{|E|}{|V|} - |E(\Omega)|,$$
which is positive for all $\Omega\subsetneqq V$ due to the condition \eqref{ns-condtion}.

It remains to prove the above claim. Let ${\Omega_t}=\{x\in V:v(x) \ge t\}$ and $v_{\min}=\min_{x\in V}v(x)$.
Notice that  $\bm{1}_{\Omega_t}\equiv1$ for any $t<v_{\min}$, 
\[\sum_{x \in V} \alpha(x) \bm{1}_{\Omega_t}(x)=\sum_{x \in V} \alpha(x)=\sum_{x \in V}\left(\frac{|E|}{|V|}- \frac{1}{2}d(x)\right)=0.\]
Therefore,
\[\int_{-\infty}^{+\infty}\sum_{x \in V} \alpha(x) \bm{1}_{\Omega_t}(x)\, dt=\sum_{x \in V} \alpha(x) \int_{v_{\min}}^{+\infty}\bm{1}_{\Omega_t}(x)\, dt =\sum_{x \in V} \alpha(x)(v(x)-v_{\min})=\sum_{x \in V} \alpha(x)v(x).\]
It follows that
\[
    \int_{-\infty}^{+\infty}\rho(\bm{1}_{\Omega_t})\, dt = \frac{1}{2}\int_{-\infty}^{+\infty} \sum_{(x,y) \in E} |\bm{1}_{\Omega_t}(x) - \bm{1}_{\Omega_t}(y)| \, dt+\int_{-\infty}^{+\infty} \sum_{x \in V} \alpha(x) \bm{1}_{\Omega_t}(x)\, dt=\rho(\bm v)
\]
by the following co-area formula 
$$ \int_{-\infty}^{+\infty} \left( \sum_{(x,y) \in E} |\bm{1}_{\Omega_t}(x) - \bm{1}_{\Omega_t}(y)| \right) \, dt=\sum_{(x,y) \in E} |v(x) - v(y)|, $$
see Lemma 3.3 in \cite{grigor2018introduction} for details. This completes the proof of the claim.  Thus, $H$ is coercive. 

According to the above results, \(H\) attains a minimum \(\bm{g}^* \in S\) with \(\nabla H(\bm{g}^*) = 0\), solving
\[
\sum_{y \sim x} \frac{1}{1 + e^{g^*(x) - g^*(y)}} = \frac{|E|}{|V|}, ~\mbox{for any}~ x\in V.
\]
Then \(m^*(x) = e^{g^*(x)}\) satisfies the equivalent system \eqref{equiv}, and \[\omega^*(x,y)= \frac{|V|}{|E|}\cdot\frac{m^*(x) m^*(y)}{m^*(x) + m^*(y)}\] solves the original equations \eqref{constant}.
\end{proof}

\begin{remark}\label{constant-cur}
Defining graph density as half of the average degree, the condition \eqref{ns-condtion} implies that the global density strictly exceeds the density of any local subgraph. Put simply, there are no clusters within the graph that are denser than the graph itself.
Here, we provide several examples to illustrate the existence and non-existence of weight for constant curvature.
\begin{enumerate}
\item Any tree $T$ admits a weight of constant curvature. Indeed, for any non-empty $\Omega\subsetneqq V$, 
$$\frac{|E(\Omega)|}{|\Omega|} \leq 1 - \frac{1}{|\Omega|} < 1 - \frac{1}{|V|} = \frac{|E|}{|V|},$$
 which implies that the condition \eqref{ns-condtion} always holds. 

    \item If a graph is not regular or semi-regular bipartite, a constant curvature weight may or may not exist. For example:
    \begin{itemize}
    \item Dumbbell graph $D_{6,6}$, see Figure (a), admits a weight of constant curvature $-\frac{2}{13}$. Since any proper subset $\Omega$ contains at most one complete cycle, $|E(\Omega)| = |\Omega|$ if it contains a cycle or $|E(\Omega)| < |\Omega|$ if it contains no cycle. Therefore, $$\max_{\emptyset\neq\Omega \subsetneqq V}\frac{|E(\Omega)|}{|\Omega|} = 1<\frac{13}{12}= \frac{|E|}{|V|}.$$
    \item Tadpole graph  $T_{6,1}$, see Figure (b),
     does not have a weight for constant curvature. Let $\Omega$ be exactly the $6$-cycle. It follows that
     $$\frac{|E(\Omega)|}{|\Omega|} = 1=\dfrac{|E|}{|V|}.$$
     \item  Heawood-Hexagon dumbbell graph, see Figure (c), does not have a weight for constant curvature. 
     Take $\Omega$ as the vertex set of the Heawood graph component, we have
     $$\frac{|E(\Omega)|}{|\Omega|} = \frac{21}{14} = 1.5>1.4= \frac{28}{20} = \frac{|E|}{|V|}. $$
    \end{itemize}
\end{enumerate}

\tikzset{
    vertex/.style={circle, fill=black, inner sep=0pt, minimum size=5pt},
    edge/.style={thick}
}

\begin{figure}[H]
    \centering
    
    \begin{subfigure}[b]{0.32\textwidth}
        \centering
        \begin{tikzpicture}[scale=0.5]
            \foreach \i in {0,60,...,300} {
                \node[vertex] (L\i) at (\i:1.5) {};
            }
            \draw[edge] (L0) -- (L60) -- (L120) -- (L180) -- (L240) -- (L300) -- (L0);
            
            \begin{scope}[shift={(4.5,0)}]
                \foreach \i in {0,60,...,300} {
                    \node[vertex] (R\i) at (\i:1.5) {};
                }
                \draw[edge] (R0) -- (R60) -- (R120) -- (R180) -- (R240) -- (R300) -- (R0);
            \end{scope}
            
            \draw[edge] (L0) -- (R180);
        \end{tikzpicture}
        \caption{Dumbbell graph $D_{6,6}$}
    \end{subfigure}
    \hfill
    \begin{subfigure}[b]{0.32\textwidth}
        \centering
        \begin{tikzpicture}[scale=0.5]
            \foreach \i in {0,60,...,300} {
                \node[vertex] (H\i) at (\i:1.5) {};
            }
            \draw[edge] (H0) -- (H60) -- (H120) -- (H180) -- (H240) -- (H300) -- (H0);
            
            \node[vertex] (Tail) at (3.0, 0) {};
            \draw[edge] (H0) -- (Tail);
            
            \path (0,-1.8) -- (0,1.8); 
        \end{tikzpicture}
        \caption{Tadpole graph $T_{6,1}$}
    \end{subfigure}
    \hfill
    \begin{subfigure}[b]{0.32\textwidth}
        \centering
        \begin{tikzpicture}[scale=0.4]
            \foreach \i in {0,...,13} {
                \node[vertex] (HW\i) at ({90 - \i*360/14}:2.5) {};
            }
            
            \foreach \i [evaluate=\i as \next using {int(mod(\i+1,14))}] in {0,...,13} {
                \draw[edge] (HW\i) -- (HW\next);
            }
            
            \foreach \i in {0, 2, ..., 12} {
                \pgfmathsetmacro{\target}{int(mod(\i+5,14))}
                \draw[edge] (HW\i) -- (HW\target);
            }

            \begin{scope}[shift={(6.0,0)}]
                \foreach \i in {0,60,...,300} {
                    \node[vertex] (Hex\i) at (\i:1.5) {};
                }
                \draw[edge] (Hex0) -- (Hex60) -- (Hex120) -- (Hex180) -- (Hex240) -- (Hex300) -- (Hex0);
            \end{scope}

            \draw[edge] (HW3) -- (Hex180); 
            
        \end{tikzpicture}
        \caption{Heawood-Hexagon dumbbell}
    \end{subfigure}
\end{figure}
\end{remark}

\begin{lemma}[Theorem 1.1 in \cite{luo2004combinatorial}]
Fix a triangulation $T$ of a closed topological surface $M$. There exists a constant discrete Gaussian curvature corresponding to PL metric associated to $T$ if and only if for any proper subset $I$ of the vertices $V$ of $T$,
\begin{equation}\label{ns_luo}
\frac{|F_I|}{|I|} > \frac{|F|}{|V|} 
\end{equation}
where $F$ is the set of all triangles in $T$ and $F_I$ is the set of all triangles having a vertex in $I$.
\end{lemma}
Next, we clarify the connection between Luo’s  condition \eqref{ns_luo} for the existence of PL metric corresponding to constant discrete Gaussian curvature and our condition \eqref{ns-condtion} for the existence of weights corresponding to constant Lin-Lu-Yau curvature. To facilitate a comparison between the two, we verify \eqref{ns_luo} on the Delaunay triangulation and  \eqref{ns-condtion} on its Voronoi diagram, and arrive at the following result.
\begin{corollary}\label{coro_sur}
     Condition \eqref{ns-condtion} is  less restrictive than Condition \eqref{ns_luo} in the dual sense. To be more precise, for any Delaunay triangulation of a surface,  \eqref{ns-condtion} is always satisfied on the dual Voronoi diagram with vertex degrees greater than 5, whereas \eqref{ns_luo} does not necessarily hold on the Delaunay triangulation itself. Notably, both conditions are satisfied in the case of 6-regular Delaunay triangulations and their Voronoi diagrams.
\end{corollary}
\begin{proof}
 For a Delaunay triangulation where each vertex has a degree greater than 5, the girth of the dual Voronoi diagram  is strictly greater than 5. Since each face (triangle) in the Delaunay triangulation corresponds to a vertex in the Voronoi diagram, and under the dual transformation, the three edges of a face map to the three edges of the dual Voronoi diagram, the Voronoi diagram is inherently a $3$-regular graph. Consequently,  from Proposition \eqref{regular}, Condition \eqref{ns-condtion} is satisfied naturally  and the weight corresponding to  the constant Lin-Lu-Yau curvature of every edge is constant, which can be regarded as the length of any edge for the Voronoi diagram. This set of edge metrics on the Voronoi diagram uniquely determines the edge metrics of the original Delaunay triangulation up to a global scaling factor. Unfortunately, this metric on the Delaunay triangulation does not necessarily correspond to a constant discrete Gaussian curvature, as Condition \eqref{ns_luo} is not always satisfied for any triangulation where each vertex has a degree greater than 5.

In particular, there exists an ideal topological configuration where both conditions \eqref{ns_luo} (for the Delaunay triangulation) and \eqref{ns-condtion} (for the dual Voronoi diagram) are simultaneously satisfied. This occurs in the case of a 6-regular triangulation, which corresponds to the topology of a flat torus. In this case, the dual Voronoi diagram is a hexagonal tiling. Starting from an arbitrary initial value, the Lin-Lu-Yau curvature flow \eqref{flow-equation1} drives the edge weights of the hexagonal diagram to converge toward a constant, the specifics of which are discussed in the next section. The constant weights can still be treated
as edge metrics. Given the strict proportionality between the Delaunay and Voronoi edge lengths in this ideal state, the Delaunay edge metric necessarily is a constant as well.
\end{proof}

\subsection{Convergence of the Ricci flow}

Given any prescribed curvature $\bm \kappa^*$ that needs to be satisfied
\begin{equation}\label{eq:conserve}
\sum_{i=1}^n\kappa_i^*=2(|V|-|E|)
\end{equation}
by Proposition \ref{pro-curvature}, we have 
\[\sum_{j=1}^n\ln \omega_j(t)=\sum_{j=1}^n\ln \omega_{j,0}=const.\]
under the evolution of the Ricci flow \eqref{case_a}.       Let $\bm r=\bm {\ln \omega}$ and the convex set
\[P=\{\bm r\in \mathbb{R}^n:\sum_{j=1}^nr_j=\sum_{j=1}^nr_{j,0}\}.\]

\begin{theorem}\label{main2}
Assume that the prescribed curvature
$\bm \kappa^*$ is attainable, i.e., there exists $\bm \omega^*$ such that $\bm \kappa(\bm \omega^*)=\bm \kappa^*$. Then the solution $\bm \omega$ to the Ricci flow \eqref{case_a} exponentially  converges to $\bm \omega^*$. And the Lin-Lu-Yau curvature exponentially converges  to $\bm \kappa^*$. 
\end{theorem}

\begin{proof} The Ricci flow \eqref{case_a} can be transformed into 
\begin{equation}\label{case_b}
    \frac{d}{dt}r_i=-(\kappa_i(e^{\bm r})-\kappa_i^*) 
\end{equation}
on  the hyperplane $P$ with an initial value $\bm r_0=(r_{1,0},r_{2,0},\cdots,r_{n,0})$. Based on the definition of $\bm \kappa$, if $\bm r^*$ is an equilibrium point to the Ricci flow \eqref{case_a}, namely $\kappa_i(e^{\bm r^*})=\kappa_i^*$ for any $i$, then $\bm r^* + c \bm{1}$ with $c \in \mathbb{R}$ is also an equilibrium point. From Proposition \ref{injective}, 
\[
L := \{\bm r\in \mathbb{R}^n: \bm r=\bm r^* + c \bm{1}, c \in \mathbb{R} \}.
\]    
is exactly the set of equilibrium points.

First, we consider the convergence of $\bm r$, and the proof is separated into three steps.\\
    \textbf{Step 1. The Ricci flow \eqref{case_b} is a gradient flow.} Let $e_i=(x,y)$, for any  $j\neq i$ and $e_j$ is incident to node $x$ or $y$,
$$\frac{\partial \kappa_{i}}{\partial r_{j}} = -\frac{2 \omega_{i}\omega_{j}}{m(x)^{2}} ~(\mbox{or} -\frac{2 \omega_{i}\omega_{j}}{m(y)^{2}}).$$  
If edge $e_j$ is incident to neither $x$ nor $y$, the above partial derivative is $0$.
Based on this, 
we may transform the Ricci flow  \eqref{case_b} into the negative  gradient flow  in the variable $\bm r$
\[\frac{d}{dt}\bm r=-\nabla \bm f, \]
 of the  function $f:P\rightarrow \mathbb{R}$ defined by line integral as 
\begin{equation*}\label{potential}
    f(\bm r):=\int_{\bm a}^{\bm r}\sum_{i=1}^n(\kappa_i-\kappa_i^*)dr_i,
\end{equation*}
where $\bm a$ is any point in $P$.\\
 \textbf{Step 2. The above function $f$ is strongly convex on \(\Sigma := \{\bm r\in P: f(\bm r) \le f(\bm r_0) \}\)  for all $\bm a\in P$.} First, we have
 \[\begin{split}
    \frac{\partial \kappa_{i}}{\partial r_{i}} &= \left[2 \left( \frac{1}{m(x)} + \frac{1}{m(y)} \right) - 2 \omega_{i}\left( \frac{1}{m(x)^{2}} + \frac{1}{m(y)^{2}} \right)\right]\cdot\omega_{i}\\
    &= \sum_{e_j \in E(x),\; j \neq i} \frac{2\omega_i \omega_j}{m(x)^2} \;+\; \sum_{e_j \in E(y),\; j \neq i} \frac{2\omega_i \omega_j}{m(y)^2},
\end{split}\]
where \( E(u) \) denotes the set of edges incident to vertex \( u \). 
 Notice that the matrix $(\frac{\partial \kappa_{i}}{\partial r_{j}})_{n\times n}$ is symmetric, satisfies \( \frac{\partial \kappa_{i}}{\partial r_{i}}> 0,\frac{\partial \kappa_{i}}{\partial r_{j}}\leq 0 \) for all $i\neq j$, and fulfills \(\sum_{i=1}^n\frac{\partial \kappa_{i}}{\partial r_{j}} = 0 \) 
 for all $j.$
Thus, for any vector \( \bm z\in \mathbb{R}^n \), 

\[
    \bm z^{\mathsf T} (\frac{\partial \kappa_{i}}{\partial r_{j}})_{n\times n} \bm z=\sum_{u \in V}\; \sum_{\{e_i,e_j\} \subseteq E(u)} \frac{2\omega_i \omega_j}{m(u)^2} \, (z_i - z_j)^2\geq 0.
\]
Therefore, the matrix $(\frac{\partial \kappa_{i}}{\partial r_{j}})_{n\times n}$ is semi-positive definite and the kernel is generated  by the vector $\bm 1.$ By Rank-Nullity Theorem, the rank of $(\frac{\partial \kappa_{i}}{\partial r_{j}})_{n\times n}$ is $n-1$. Moreover, we perform an eigenvalue decomposition of the Hessian matrix of $f$
\begin{equation*}\label{egi-decom}
    (\frac{\partial \kappa_{i}}{\partial r_{j}})_{n\times n} =  \lambda_2 \bm{v}_2 \bm{v}_2^T + \cdots + \lambda_n \bm{v}_n \bm{v}_n^T,
\end{equation*}
where \( 0=\lambda_1<\lambda_2\leq \cdots\leq \lambda_n \) are eigenvalues and $\bm v_i$ are corresponding eigenvectors. Let  \(T_{\bm{r}}P \) be the tangent space of $P$.
For any $\bm{u} \neq \bm{0}, \bm u \in T_{\bm{r}}P $, that is, $\bm u\perp \bm 1$, we have
\(\|\bm u\|^2=\sum_{i=2}^n (\bm{v}_i^T \bm{u})^2>0,\)
thus 
\[
\bm{u}^T (\frac{\partial \kappa_{i}}{\partial r_{j}})_{n\times n} \bm{u} = \sum_{i=2}^n \lambda_i (\bm{v}_i^T \bm{u})^2\geq \lambda_2 \|\bm u\|^2(>0).
\]
Therefore, $f$ restricted on $P$ is strictly convex.

Next, we {\it claim} that $f$ is coercive on $P$ when $\kappa_i^*$ is attainable, i.e., there exists $r^*\in P$ such that $\kappa_i(\bm r^*)=\kappa_i^*$ for any $i\in E$. This condition is equivalent to $\nabla f(\bm r^*) = \bm 0$ by virtue to Step 1.  By the strict convexity of the $f(\bm r)$, $\bm r^*$ is unique, which is also the unique intersection point of the set of equilibrium points \( L \) and the hyperplane \( P \).
For any $\tau  \ge 0$ and $\bm 0\neq \bm \phi\in \mathbb{R}^n$ such that $\bm \phi \perp \bm 1$, define 
$g(\tau,\bm \phi) = f(\bm r^* + \tau \bm \phi)$.
We obtain 
$$\partial_\tau g(\tau,\bm \phi) = \langle \nabla f(\bm r^* + \tau \bm\phi), \bm\phi \rangle.$$
In particular, we have
$\partial_\tau g(0,\bm\phi)  = \langle \nabla f(\bm r^*), \bm\phi \rangle  = 0.$
Moreover, for all $\tau > 0$,
$$\partial^2_\tau g(\tau,\bm \phi) = \bm \phi^T \nabla^2 f(\bm r^* + \tau \bm \phi) \bm \phi>0$$
by the strictly convexity of $f.$
Then for all $\tau > 0$,
$\partial_\tau g(\tau,\bm \phi) > 0.$
Moreover,
for any $\tau>\tau_0>0$, 
$$g(\tau,\bm \phi) =g(\tau_0,\bm \phi) +\int_{\tau_0}^{\tau} \partial_sg(s,\bm \phi) ds\ge g(\tau_0,\bm \phi) + \partial_sg(\tau_0,\bm \phi) (\tau-\tau_0).$$
Since $f(\bm r)\in C^2$, both $g(\tau_0,\bm \phi)$ and $ \partial_sg(\tau_0,\bm \phi)$ are continuous functions with respect to $\bm \phi$. Let  $\mathbb{S}^{n-1} = \{ \bm u \in \mathbb{R}^{n} : \|\bm u\| = 1 \}$. Notice that $\mathbb{S}^{n-1} $ is compact. It follows that, on $\mathbb{S}^{n-1}$, $g(\tau_0,\bm \phi)$ has a minimum  $M_1$, and $\partial_sg(\tau_0,\bm \phi)$ has a minimum  $M_2>0$. Therefore,
\begin{equation}\label{estimate}
    g(\tau,\bm \phi) \geq M_1+M_2 (\tau-\tau_0).
\end{equation}
For any $\bm r\in P$,  we choose
$$\tau = \|\bm r - \bm r^*\|,~ \mbox{and}~ \bm \phi= \frac{\bm r - \bm r^*}{\|\bm r - \bm r^*\|}.$$
Obviously,  $\|\bm \phi\| = 1$, $\bm \phi \perp \bm 1$ and $\bm r = \bm r^* + \tau \bm \phi$. Then by \eqref{estimate},
$$f(\bm r) \ge M_1 +M_2 (\|\bm r - \bm r^*\| - \tau_0)\rightarrow +\infty$$
as $\|\bm r\|\rightarrow \infty$, which implies that $f(\bm r)\rightarrow +\infty.$ This finishes the proof of the claim.

From Step 1, we have
$$\frac{d}{dt} f(r(t)) = \langle \nabla f(r), \frac{dr}{dt} \rangle= -\langle \nabla f(r), \nabla f(r) \rangle =- \|\nabla f(r)\|^2\leq 0.$$
Notice that $\Sigma$ is compact  since $f$ is continuous and coercive towards to $+\infty$ with respect to $\bm r.$ Then $\lambda_2(\bm r)$ attains its global minimum on $\Sigma$, denoted as 
$$\lambda = \min_{(\bm r,\bm \phi) \in \Sigma \times\mathbb{S}^{n-1} } \bm \phi^T \nabla^2 f(\bm r) \bm \phi>0.$$
This yields that $f(\bm r)$ is strongly convex on $\Sigma$.

\noindent \textbf{Step 3. The Ricci flow \eqref{case_b} convergences to $r^*$ exponentially.} 
Define the Lyapunov function as
\[
V(t) = \frac{1}{2} \|\bm r(t) - \bm r^*\|^2.
\]
Differentiating with respect to time to get
\[
\frac{dV}{dt} = (\bm r - \bm r^*)^T \frac{d\bm r}{dt} = -(\bm r - \bm r^*)^T \nabla \bm f(\bm r).
\]
Since both \(\bm r(t)\) and \(\bm r^*\) lie in the hyperplane \(P\), their difference vector \(\bm r - \bm r^*\) must belong to the tangent space \(T_{\bm{r}}P \).  Define 
\[g(\bm r):=f(\bm r)-\frac{\lambda}{2}\|\bm r\|^2.\]
Given the strong convexity of $f$ on $P$, $g(\bm r)$ is convex on $P$ because the Hessian matrix of $g$ is positive semi-definite, i.e.,
\[\bm u^T\nabla^2 g\bm u=\bm u^T\nabla^2 f\bm u-\lambda\|\bm u\|^2\geq 0\]
for any $\bm{u} \neq \bm{0}, \bm u \in T_{\bm{r}}P $.
By the monotonicity of the gradient of the convex function, we have
$$(\nabla g(\bm r) - \nabla g(\bm r^*))^T(\bm r - \bm r^*)\ge 0.$$
Together with $\nabla g(\bm r)=\nabla f(\bm r)-\lambda\bm r$, this yields
\[
\bigl(\nabla f(\bm r) - \nabla f(\bm r^*)\bigr)^T (\bm r - \bm r^*) \geq \lambda \|\bm r - \bm r^*\|^2.
\]
where \(\nabla \bm f(\bm r^*) = 0\). Therefore, we have
\[
\frac{dV}{dt} \leq -\lambda \|\bm r - \bm r^*\|^2 = -2\lambda V(t).
\]
Applying Gronwall's inequality yields
\[
V(t) \leq V(0) e^{-2\lambda t},
\]
which gives
\[
\|\bm r(t) - \bm r^*\| \leq \|\bm r(0) - \bm r^*\| e^{-\lambda t}.
\]
Since both the weights \(\bm \omega(=e^{\bm r})\) and the curvatures \(\bm \kappa\) (see Lemma \ref{local_lip}) are local Lipschitz functions of \(\bm r\), it follows that the weights  and the curvatures  also converge exponentially.

\end{proof}
Let $e_i=(x,y)$, the evolution equation for the curvature satisfies
 \[\frac{d}{dt}\kappa_i(t)=\sum_{e_j\in E(x),j\neq i}\frac{2\omega_i\omega_j}{m^2(x)}[(\kappa_j-\kappa^*_j)-(\kappa_i-\kappa^*_i)]+\sum_{e_j\in E(y),j\neq i}\frac{2\omega_i\omega_j}{m^2(y)}[(\kappa_j-\kappa^*_j)-(\kappa_i-\kappa^*_i)].\]
By abuse of notation, we also use the symbol $j\sim i$ to indicate that two edges $e_i$ and $e_j$ intersect at $x$ or $y$. Then 
\[\frac{d}{dt}\kappa_i(t)=
    \sum_{j\sim i}C_{ij}[(\kappa_j-\kappa^*_j)-(\kappa_i-\kappa^*_i)],\]  
where 
\[C_{ij}:=
 \begin{cases}
\frac{2\omega_i\omega_j}{m^2(x)},& \mbox{if}~ e_j\in E(x),\\
\frac{2\omega_i\omega_j}{m^2(y)},&\mbox{if}~ e_j\in E(y),\\
0,&\mbox{otherwise.}
\end{cases}
  \] 
Notice that $C_{ij}=C_{ji}.$

\begin{theorem}\label{main3}
   If a solution of the Ricci flow \eqref{case_a} with the prescribed curvature $\bm\kappa^*$ converges, then  weights of $\bm\kappa^*$ must exist. Moreover, the solution necessarily converges to  this prescribed curvature weight.
\end{theorem}
\begin{proof}
Consider 
\(g(t):=\sum_{i=1}^n(\kappa_i-\kappa_i^*)^2\), we have 
\[\begin{split}
    g'(t)&=2\sum_{i=1}^n(\kappa_i-\kappa_i^*)\kappa_i'(t)\\
    &=2\sum_{i=1}^n\sum_{j\sim i}C_{ij}(\kappa_i-\kappa_i^*)[(\kappa_j-\kappa^*_j)-(\kappa_i-\kappa^*_i)]\\
    &=\sum_{i=1}^n\sum_{j\sim i}C_{ij}(\kappa_i-\kappa_i^*)[(\kappa_j-\kappa^*_j)-(\kappa_i-\kappa^*_i)]+\sum_{i=1}^n\sum_{j\sim i}C_{ij}(\kappa_j-\kappa_j^*)[(\kappa_i-\kappa^*_i)-(\kappa_j-\kappa^*_j)]\\
    &=-\sum_{i=1}^n\sum_{j\sim i}C_{ij}[(\kappa_j-\kappa^*_j)-(\kappa_i-\kappa^*_i)]^2\\
    &=(\bm \kappa-\bm \kappa^*)^T\bm J(\bm \kappa-\bm \kappa^*),
\end{split}\]
where $\bm J:=(\frac{\partial\kappa_i}{\partial r_j})_{n\times n}$,
 and $\mbox{Ker}(\bm J)=\mbox{span}(1,1,\cdots,1).$

Since the weight $\bm \omega$ converges, set
\(\lim_{t\rightarrow+\infty}\bm \omega(t)=\bm \omega(+\infty),\)
then the curvature $\bm \kappa$ and $\bm J$ also converge to
$\bm \kappa(+\infty)$ and $\bm J(+\infty)$, respectively. Precisely,
\[\bm \kappa(+\infty)=\bm \kappa(\bm\omega(+\infty)), \quad \bm J(+\infty)=\bm J(\bm\omega(+\infty)).\]
Thus, both \(g\) and  \(g'\) converge. It follows that 
\[\lim_{t\rightarrow +\infty}g'(t)=0,\] 
which implies that
\[(\bm \kappa(+\infty)-\bm \kappa^*)^T\bm J(+\infty)(\bm \kappa(+\infty)-\bm \kappa^*)=0.\]
Notice that $\bm J(+\infty)$ is also a positive semi-definite matrix, and  $\mbox{Ker}(\bm J(+\infty))=\mbox{span}(1,1,\cdots,1).$ It follows that 
  \begin{align*}
      (\bm \kappa(+\infty)-\bm \kappa^*)^T\bm J(+\infty)(\bm \kappa(+\infty)-\bm \kappa^*)
      &= (\bm \kappa(+\infty)-\bm \kappa^*)^T\bm J^{1/2}(+\infty)\bm J^{1/2}(+\infty)(\bm \kappa(+\infty)-\bm \kappa^*)\\
      &= \|\bm J^{1/2}(+\infty)(\bm \kappa(+\infty)-\bm \kappa^*)\|^2.
  \end{align*}
If \(\|\bm J^{1/2}(+\infty)(\bm \kappa(+\infty)-\bm \kappa^*)\|^2=0\), then necessarily \( \bm J^{1/2}(+\infty)(\bm \kappa(+\infty)-\bm \kappa^*) = \bm0 \).  
  Left‑multiplying both sides by \( \bm J^{1/2}(+\infty)\) gives  
  \[
  \bm J(+\infty)(\bm \kappa(+\infty)-\bm \kappa^*) = \bm0,
  \]  
which means $\bm \kappa(+\infty)-\bm \kappa^*\in \mbox{Ker}(\bm J(+\infty)).$  That is, there exists $c\in \mathbb{R}$ such that 
$$\bm \kappa(+\infty)=\bm \kappa^*+c\bm 1.$$
According to  \eqref{eq:conserve} and Proposition \ref{pro-curvature}, we have 
\[\sum_{e\in E}\kappa_e=\sum_{e\in E}\kappa_e^*=2(|V|-|E|).\]
It follows that $c=0$, i.e.,
\[\bm \kappa(\bm\omega(+\infty))=\bm \kappa^*.\] 
From Proposition \ref{injective}, $\bm\omega(+\infty)$ is the unique weight, up to a scalar multiplication, realizing $\bm \kappa^*$.
\end{proof}

By choosing $\kappa_i^*=\bar{\kappa}$ in the Ricci flow \eqref{case_a}, we have
\begin{equation}\label{case_av} 
    \begin{cases}
\frac{d}{dt}\omega_i=-(\kappa_i(\bm \omega)-\bar{\kappa})\omega_i  \\
\bm \omega(0)=\bm \omega_0
\end{cases}
\end{equation}
with $\bm \omega_0>\bm0$. From Theorem \ref{condition_th}, Theorem \ref{main2} and Theorem \ref{main3}, we have the following conclusion.

\begin{corollary}\label{main_coro}
      The Ricci flow \eqref{case_av} converges exponentially to the weight $\bm \omega$ of constant curvature $\bar{\kappa}$ if and  only if the condition \eqref{ns-condtion} holds. 
\end{corollary}
\begin{remark}\label{rk2}
Notice that
$\omega_i(t)$ is the solution to the equations \eqref{case_a}  if and only if $e^{-t{\kappa^*_i}}\omega_i(t)$ is the solution to the original form of Ricci flow
\begin{equation}\label{case_classical}
    \frac{d}{dt}\omega_i=-\kappa_i(\bm \omega)\omega_i.
\end{equation}
Then, $\omega_i(t)$ is the solution to the equations \eqref{case_a} if and only if
$\frac{e^{-t\kappa_i^*}\omega_i(t)}{\sum_{j=1}^ne^{-t\kappa_j^*}\omega_j(t)}$ is the solution to the normalized Ricci flow 
\begin{equation}\label{case_nor} 
\frac{d}{dt}\omega_i=-\omega_i\kappa_i(\bm \omega)+\omega_i\sum_{j=1}^n\kappa_j(\bm \omega) \omega_j 
\end{equation}
\end{remark}
As a direct application, by Remark \ref{rk2}, the normalized Ricci flow \eqref{case_nor} converges exponentially to  the weight $\bm \omega/ \bm \omega^T\bm 1$, and the curvature under the evolution of \eqref{case_classical} and \eqref{case_nor} converges exponentially to $\bar{\kappa}$ if and  only if the condition \eqref{ns-condtion} holds.

\section{Applications and simulations}\label{simulations}

Given that the existence and uniqueness guaranties do not require the prescribed curvature to be strictly attainable, see Theorem \ref{main1}, we may assume it to be zero, effectively recovering the original Ricci flow \eqref{flow-equation0}. 
As demonstrated in prior research, this approach is applicable to tasks such as community detection, network robustness analysis, and core detection. 
Since the algorithms are very similar to established methods \cite{2022Normalized,MY24}, bypassing the need to compute shortest paths between nodes, we omit further details here.

Instead, this section focuses on the prescribed constant curvature flow \eqref{case_av} on graphs with girth of at least $6$. This flow drives the edge weights to ensure that the Lin-Lu-Yau curvature of each edge converges to a constant, ultimately leading to a state of geometric uniformization. Notably, for regular and semi-regular bipartite graphs, the weights yielding constant curvature are themselves constant (see Proposition \ref{regular}). In contrast, the landscape for general graphs is significantly more complex, as constant curvature weights may not necessarily exist (see Remark \ref{constant-cur}). These concepts find diverse applications: the Ricci flow \eqref{case_av} can be utilized to identify network bottlenecks and structural anomalies, as well as to determine optimal edge weights. Furthermore, it can be applied to recover the geometric embedding of a surface tessellation.

\subsection{Network structure detection and optimal weight assignment}

By adjusting edge weights, curvature flow eliminates curvature disparities within a network, driving it toward a state of constant curvature. These weights can be interpreted as edge capacities. In this context, geometric uniformization is equivalent to an optimal capacity-structure alignment. That is, based on the network's topological characteristics, the flow precisely determines the ideal capacity for each path or route, ensuring that the stress is distributed uniformly across the entire system when facing traffic surges. 

In graphs that are neither regular nor semi-regular bipartite, to enforce curvature homogenization,  the Ricci flow  \eqref{case_av} must drive the weights at bottlenecks to values significantly larger than those of other edges. Figure 1 illustrates a clear and intuitive mechanism for network bottleneck identification based on \eqref{case_av}, with unit initial weights assigned to all edges. The evolutionary process forces a stratification of edge weights to achieve geometric curvature homogenization. Specifically, 
\begin{itemize}
    \item the weights of the red curve surge significantly ($\omega>2.5$), driven by Lin-Lu-Yau curvature that remains far below the target constant curvature, i.e., $\kappa - \bar{\kappa}<0$; 
    \item the weights represented by blue and green curves  remain relatively low ($\omega<1.0$). 
\end{itemize}
This pronounced disparity in weights allows for the precise localization of bottlenecks using a simple threshold (e.g., $\mbox{Threshold}= 2.0$). 

To explore the advantages of the Ricci flow \eqref{case_av} in detecting graph asymmetry, we perform vertex addition and edge deletion on the Möbius-Kantor graph ($GP(8, 3)$) to obtain an asymmetric graph. Specifically, we introduce two localized perturbations: the addition of node 16 into edge $(0,1)$ and the removal of the edge between nodes 5 and 6, see Figure 2. 
As the evolution seeks geometric uniformization over this asymmetric graph, a distinct stratification of edge weights $\omega$ emerges:

\definecolor{edgeRed}{RGB}{255, 0, 0}      
\definecolor{edgeOrange}{RGB}{255, 165, 0} 
\definecolor{edgeGreen}{RGB}{0, 128, 0}    
\definecolor{edgeBlue}{RGB}{0, 0, 255}     
\definecolor{nodeGray}{RGB}{200, 200, 200}

\begin{figure}[htbp]
    \centering
    
    \begin{tikzpicture}[scale=0.9, transform shape]
        \tikzstyle{myvertex}=[circle, fill=nodeGray, draw=black, thick, minimum size=6mm, inner sep=0pt, font=\bfseries\small]
        \tikzstyle{myedge}=[line width=2pt]

        \coordinate (L_center) at (-3, 0);
        \coordinate (R_center) at (3, 0);
        \def\R{1.8} 

        \node[myvertex] (0) at (-1.2, 0) {0};
        \node[myvertex] (1) at (-2.1, 1.55) {1};
        \node[myvertex] (2) at (-3.9, 1.55) {2};
        \node[myvertex] (3) at (-4.8, 0) {3};
        \node[myvertex] (4) at (-3.9, -1.55) {4};
        \node[myvertex] (5) at (-2.1, -1.55) {5};

        \node[myvertex] (6) at (1.2, 0) {6};
        \node[myvertex] (7) at (2.1, 1.55) {7};
        \node[myvertex] (8) at (3.9, 1.55) {8};
        \node[myvertex] (9) at (4.8, 0) {9};
        \node[myvertex] (10) at (3.9, -1.55) {10};
        \node[myvertex] (11) at (2.1, -1.55) {11};

        \draw[myedge, edgeRed] (0) -- (6);
        \draw[myedge, edgeOrange] (0) -- (1); \draw[myedge, edgeOrange] (0) -- (5);
        \draw[myedge, edgeOrange] (6) -- (7); \draw[myedge, edgeOrange] (6) -- (11);
        \draw[myedge, edgeGreen] (1) -- (2); \draw[myedge, edgeGreen] (4) -- (5);
        \draw[myedge, edgeGreen] (7) -- (8); \draw[myedge, edgeGreen] (10) -- (11);
        \draw[myedge, edgeBlue] (2) -- (3); \draw[myedge, edgeBlue] (3) -- (4);
        \draw[myedge, edgeBlue] (8) -- (9); \draw[myedge, edgeBlue] (9) -- (10);


    \end{tikzpicture}

    \vspace{0.5cm}

    \begin{minipage}{0.48\textwidth}
        \centering
        \begin{tikzpicture}
            \begin{axis}[
                width=\linewidth,
                height=6cm,
                title={Edge Weights Evolution ($T=30$)},
                xlabel={Time ($t$)},
                ylabel={Weight ($\omega$)},
                grid=major,
                grid style={opacity=0.3},
                ymin=0.5, ymax=2.8,
                xmin=0, xmax=30
            ]
                \addplot[edgeRed, very thick, domain=0:30, samples=100] {2.7 - 1.7*exp(-0.35*x)};
                
                \addplot[edgeOrange, thick, domain=0:30, samples=100] {1.56 - 0.56*exp(-0.35*x)};
                
                \addplot[edgeGreen, thick, domain=0:30, samples=100] {0.83 + 0.17*exp(-0.35*x)};
                
                \addplot[edgeBlue, thick, domain=0:30, samples=100] {0.6 + 0.4*exp(-0.35*x)};
            \end{axis}
        \end{tikzpicture}
    \end{minipage}
    \hfill
    \begin{minipage}{0.48\textwidth}
        \centering
        \begin{tikzpicture}
            \begin{axis}[
                width=\linewidth,
                height=6cm,
                title={Curvature Convergence to $\bar{\kappa}$ ($T=30$)},
                xlabel={Time ($t$)},
                ylabel={Ricci Curvature ($\kappa$)},
                grid=major,
                grid style={opacity=0.3},
                ymin=-0.7, ymax=0.05,
                xmin=0, xmax=30
            ]
                \def\target{-0.154}

                \addplot[edgeRed, very thick, domain=0:30, samples=100] {\target - 0.52*exp(-0.4*x)};

                \addplot[edgeOrange, thick, domain=0:30, samples=100] {\target - 0.18*exp(-0.4*x)};

                \addplot[edgeGreen, thick, domain=0:30, samples=100] {\target + 0.12*exp(-0.4*x)};

                \addplot[edgeBlue, thick, domain=0:30, samples=100] {\target + 0.154*exp(-0.4*x)};
                
                \addplot[black!80, dashed, thick, domain=0:30] {\target};
            \end{axis}
        \end{tikzpicture}
    \end{minipage}

    
    \vspace{0.3cm}
 \caption*{\small Figure 1: The evolution on $D_{6,6}$}
    \label{fig:evolution_d66}
\end{figure}

\definecolor{cRed}{RGB}{231, 76, 60}      
\definecolor{cOrange}{RGB}{243, 156, 18}  
\definecolor{cPurple}{RGB}{142, 68, 173}  
\definecolor{cBlue}{RGB}{52, 152, 219}    
\definecolor{cGray}{RGB}{149, 165, 166}   

\begin{tikzpicture}[scale=0.92, transform shape]
    \def\Rout{2.5} 
    \def\Rin{1.5}  
    
    \begin{scope}[shift={(-4.5, 0)}] 
        
        \node[font=\small\bfseries, align=center] at (0, 3.5) {(a)  Möbius-Kantor GP(8,3)};

        \foreach \i in {0,...,7} {
            \pgfmathsetmacro{\ang}{90 + \i*360/8}
            \coordinate (O\i) at (\ang:\Rout);
            \coordinate (I\i) at (\ang:\Rin);
        }
        
        \foreach \i in {0,...,7} {
            \pgfmathtruncatemacro{\next}{mod(\i+1,8)}
            \pgfmathtruncatemacro{\innerNext}{mod(\i+3,8)}
            
            \draw[gray!80, thick] (O\i) -- (O\next);
            \draw[gray!80, thick] (O\i) -- (I\i);
            \draw[gray!80, thick] (I\i) -- (I\innerNext);
        }
        
        \foreach \i in {0,...,7} {
            \node[circle, fill=cGray!50, draw=gray, inner sep=0pt, minimum size=13pt, font=\scriptsize\bfseries] at (O\i) {\i};
            \pgfmathtruncatemacro{\inId}{\i+8}
            \node[circle, fill=cGray!50, draw=gray, inner sep=0pt, minimum size=13pt, font=\scriptsize\bfseries] at (I\i) {\inId};
        }
    \end{scope}

    \begin{scope}[shift={(4.5, 0)}] 

        \node[font=\small\bfseries, align=center] at (0, 3.5) {(b) Asymmetric graph based on GP(8,3)};

        \foreach \i in {0,...,7} {
            \pgfmathsetmacro{\ang}{90 + \i*360/8}
            \coordinate (O\i) at (\ang:\Rout);
            \coordinate (I\i) at (\ang:\Rin);
        }
        \pgfmathsetmacro{\midang}{90 + 0.5*360/8}
        \coordinate (Transit) at (\midang:\Rout*1.05);

        
        \foreach \i in {1,2,3,4,6} {
            \pgfmathtruncatemacro{\next}{mod(\i+1,8)}
            \draw[cBlue!60, thick] (O\i) -- (O\next);
        }
        \foreach \i in {0,...,7} {
            \pgfmathtruncatemacro{\innerNext}{mod(\i+3,8)}
            \draw[cBlue!60, thick] (I\i) -- (I\innerNext);
        }
        \foreach \i in {0,...,7} {
             \draw[cBlue!60, thick] (O\i) -- (I\i);
        }

        \draw[cRed, line width=2.5pt] (O0) -- (Transit) -- (O1);

        \draw[cOrange, line width=2.5pt] (O5) -- (O4);
        \draw[cOrange, line width=2.5pt] (O5) -- (I5); 
        \draw[cOrange, line width=2.5pt] (O6) -- (O7);
        \draw[cOrange, line width=2.5pt] (O6) -- (I6);

        \draw[cPurple, line width=2.5pt] (O0) -- (O7);

        \draw[gray, dashed, thick] (O5) -- (O6);

        
        \foreach \i in {0,1,2,3,4,7,8,9,10,11,12,13,14,15} { 
             \ifnum\i<8 
                \node[circle, fill=cGray!30, text=black, inner sep=0pt, minimum size=14pt, font=\scriptsize\bfseries] at (O\i) {\i};
             \else
                \pgfmathtruncatemacro{\idx}{\i-8}
                \node[circle, fill=cGray!30, text=black, inner sep=0pt, minimum size=14pt, font=\scriptsize\bfseries] at (I\idx) {\i};
             \fi
        }
        
        \node[circle, fill=cOrange, text=white, inner sep=0pt, minimum size=15pt, font=\scriptsize\bfseries] at (O5) {5};
        \node[circle, fill=cOrange, text=white, inner sep=0pt, minimum size=15pt, font=\scriptsize\bfseries] at (O6) {6};
        \node[circle, fill=cRed, text=white, inner sep=0pt, minimum size=16pt, font=\scriptsize\bfseries] at (Transit) {16};

    \end{scope}

    \node (sim_image) [anchor=north] at (0, -3.0) {
        \includegraphics[width=17.5cm]{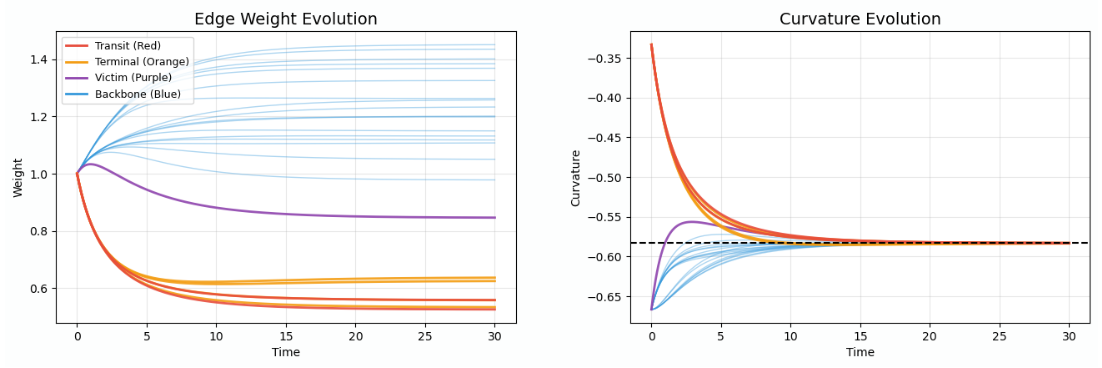}
    };

    \node[font=\small, align=center] at ($(sim_image.south) + (0, -0.3)$) {
        Figure 2: The evolution on the asymmetric graph based on GP(8,3). 
    };
\end{tikzpicture}

\begin{itemize}
    \item the weights of red and orange edges undergo rapid decay ($\omega < 0.7$), reflecting the flow's tendency to compress anomalous structures; 
    \item  the weights of the majority of blue edges expand ($\omega > 1.2$), effectively reinforcing the network's backbone;
    \item the weight of
the purple edge ((0-7)  serves as a representative edge) exhibits non-monotonic dynamics, rising initially before decreasing to a value below that of the blue edges, thereby revealing the effects induced by the anomalous structures.

\end{itemize}
By the stratification of edge weights,  topological defects are explicitly captured.

\subsection{Recovering the geometric embedding of a surface tessellation}
We perform a polygonal tiling  on the surface, requiring the girth of its 1-skeleton to be greater than 5. The edge weights are defined as the geodesic distances between endpoints, and the Ricci flow \eqref{case_av} is evolved on this 1-skeleton.
According to Proposition \ref{regular}, the constant curvature weights for regular or  semi-regular bipartite graphs are themselves constant. Consequently,  the evolution of \eqref{case_av} continuously diminishes the disparities between the edge weights, eventually reaching constant edge lengths without the need to monitor the polygon closure condition during the evolutionary process.

In Figure 3, we perform a tiling on a torus using $GP(8,3)$, with initial weights (namely, the initial edge lengths) assigned by random variables. By employing this flow \eqref{case_av}, a set of weights effectively ``inflates'' the initially ``flabby'' topological structure of the torus into a geometric form characterized by perfect symmetry. This allows for the measurement of the torus's moduli parameters, thereby determining its conformal structure. 


%
\begin{figure}[H] 
    \centering
    \includegraphics[width=1\textwidth]{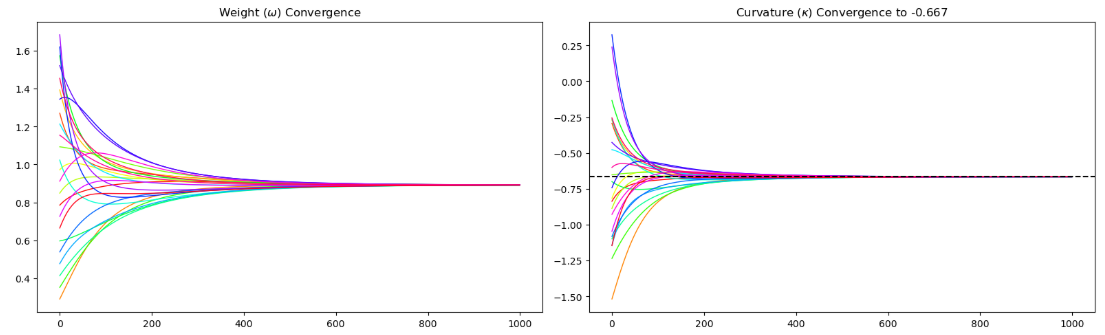}
    \caption*{Figure 3: The evolution on GP(8,3) with initial weights assigned by random variables}
    \label{fig:ricci_evolution}
    
    \small 
    \begin{flushleft}
    \end{flushleft}
\end{figure}


{\bf Acknowledgments:} 
Y. Lin is supported by NSFC, no.12471088. S. Liu is supported by NSFC, no.12001536, 12371102.

\bibliographystyle{plain}
\bibliography{citations}

\end{document}